\documentclass{amsart}

\usepackage{lineno,hyperref}
\modulolinenumbers[5]

\usepackage{amssymb,amsmath}
\usepackage{amsthm}

\newcommand{\real}{\ensuremath{\mathbb{R}}}
\newcommand{\rn}{\ensuremath{\mathbb{R}^n}}

\newcommand{\nat}{\ensuremath{\mathbb{N}}}
\newcommand{\no}{\ensuremath{\nat_0}}

\newcommand{\zn}{\ensuremath{\mathbb{Z}^n}}
\newcommand{\dint}{\ensuremath{\mathrm{d}}}
\newcommand{\supp}{\ensuremath{\operatorname{supp}}}
\newcommand{\id}{\ensuremath{\operatorname{id}}}
\newcommand{\rank}{\ensuremath{\operatorname{rank}}}

\newcommand{\spa}{\ensuremath{\operatorname{span}}}

\newcommand{\dom}{\ensuremath{\operatorname{dom}}}
\newcommand{\Lloc}{\ensuremath{L_1^{\mathrm{loc}}}}
\newcommand{\Com}{\ensuremath{C^m_0}}

\newcommand{\Emp}{\ensuremath{E^m_{p,\sigma}(B)}}

\newcommand{\Dd}{\ensuremath{\mathrm{D}}}
\newcommand{\vol}{\ensuremath{\operatorname{vol}}}

\newcommand{\bli}{\begin{list}{}{\labelwidth1.7em\leftmargin2.1em}}
\newcommand{\eli}{\end{list}}
\newtheorem{lemma}{Lemma}[section]

\newtheorem{prop}[lemma]{Proposition}
\newtheorem{theo}[lemma]{Theorem}

\theoremstyle{definition}
\newtheorem{defi}[lemma]{Definition}

\theoremstyle{remark}
\newtheorem{rem}[lemma]{Remark}

\numberwithin{equation}{section}

\begin{document}

%\begin{frontmatter}

\title[Entropy and approximation numbers via bracketing]{Entropy and approximation numbers of weighted Sobolev spaces via bracketing}

\author{Therese Mieth}
\address{Institut f\"ur Mathematik, Fakult\"at f\"ur Mathematik und Informatik, Friedrich-Schiller-Universit\"at Jena, 07737 Jena, Germany}
\email{therese.mieth@uni-jena.de}

\begin{abstract}
We investigate the asymptotic behaviour of entropy and approximation numbers of the compact embedding $\id: E^m_{p,\sigma}(B)\hookrightarrow L_p(B)$, $1\leq p<\infty,$ defined on the unit ball $B$ in $\rn$. Here $E^m_{p,\sigma}(B)$ denotes a Sobolev space with a power weight perturbed by a logarithmic function. The weight contains a singularity at the origin. Inspired by Evans and Harris \cite{EH93}, we apply a bracketing technique which is an analogue to that of Dirichlet-Neumann-bracketing used by Triebel in \cite{Tri12} for $p=2.$
\end{abstract}


%\end{frontmatter}

%\linenumbers

\maketitle

\section{Introduction}
This paper deals with the compact embedding 
\begin{equation}\label{emb}
\id: \Emp\hookrightarrow L_p(B),\hspace{1cm} 1\leq p<\infty, m\in\nat, \sigma>0,
\end{equation}
where $E^m_{p,\sigma}(B)$ is the closure of $C_0^m(B)$ with respect to the norm
\begin{equation}
\Vert f|\Emp\Vert := \left( \int_B |x|^{mp}(1+|\log|x||)^{\sigma p} \sum_{|\alpha|=m}|\Dd^\alpha f(x)|^p
\dint x\right)^{1/p}
\end{equation}
and $B=\{x\in\rn: |x|<1\}$ is the unit ball. Problems of this type have been considered so far in \cite{Tri12, Mi15a, Mi15b}. In case of Hilbert spaces Triebel obtained in \cite{Tri12} sharp results for the corresponding entropy and approximation numbers. Namely, if $p=2$ it holds for $k\in\nat,k\geq 2,$ that
\begin{equation}\label{1.2}
a_k(\id)\sim e_k(\id)\sim
\begin{cases} 
k^{-\frac{m}{n}}\hspace{1cm}&  \text{ if } \sigma>\frac{m}{n}\\
{k}^{-\frac{m}{n}}\,(\log k)^{\frac{m}{n}}\hspace{1cm}&  \text{ if } \sigma=\frac{m}{n}\\
k^{-\sigma}\hspace{1cm}&  \text{ if } 0<\sigma<\frac{m}{n}.
\end{cases}
\end{equation}
We will extend this result and confirm Triebel's Conjecture 3.8 in \cite{Tri12} that \eqref{1.2} holds for all $1\leq p<\infty.$ In \cite{Tri12} the so called Courant-Weyl method of Dirichlet-Neumann-bracketing is used. This technique is not available for $p\neq2$, but a partial analogue was established by Evans and Harris in \cite{EH93}. These authors deal with Sobolev spaces $W^1_p(\Omega)$ on a wide class of domains, i.e. rooms and passages domains or generalised ridged domains. We want to transfer this idea to control the singularity in our situation.

The paper is organised as follows. In Section \ref{sec_prelim} we collect basic notation and briefly introduce the setting of the compact embedding \eqref{emb}.
In Section \ref{sec_ak} we present a bracketing method to determine the asymptotic behaviour of the number
\begin{equation*}
\nu_0(\varepsilon, \Omega):= \max \{k \in\nat: a_k(\id_\Omega)\geq \varepsilon\},\hspace{1cm} \varepsilon>0,
\end{equation*}
as $\varepsilon\to 0$. The operator $\id_\Omega$ denotes the restriction of $\id$ to a subset $\Omega\subset B.$  Let $\Omega=\Big(\bigcup\limits_{j=1}^J \overline{\Omega}_j\Big)^{\circ}$ with disjoint domains $\Omega_j$. The essence of the method is the bracketing property stated in Proposition \ref{prop1}, which reads as
\begin{equation}\label{intro_brack}
\sum\limits_{j=1}^J \mu_0(\varepsilon, \Omega_j)
\leq
\nu_0(\varepsilon, \Omega)
\leq 
\sum\limits_{j=1}^J \nu_0(\varepsilon, \Omega_j).
\end{equation}
Here the number $\mu_0(\varepsilon, \Omega),\,\Omega\subseteq B,$ is defined by
\begin{equation*}
\mu_0(\varepsilon, \Omega)
:= 
\max\, \Big\{\,\dim S:\alpha(S)= \sup_{u\in S\setminus\{ 0\}}\dfrac{\Vert u| E^m_{p,\sigma}(\Omega)\Vert}{\Vert u| L_p(\Omega)\Vert}\leq \frac{1}{\varepsilon}\,\Big\}. 
\end{equation*}
where the maximum is taken over all finite-dimensional linear subspaces $S$ of $E^m_{p,\sigma}(\Omega)$. This approach attempts to mimic the Dirichlet-Neumann bracketing described in \cite{Mi15a,Mi15b,Tri12} if $p=2$. This is emphasised by Proposition \ref{p=2}. We use \eqref{intro_brack} to cut off the singularity of the weight $b_{m,\sigma}(x)=|x|^{mp}(1+|\log|x||)^{\sigma p},\ x\in B,$ at the origin and consider the corresponding domains separately. In this way we achieve in Proposition \ref{nu_0}
\begin{equation}
\nu_0(\varepsilon, B)
\sim
\begin{cases} 
\varepsilon^{-\frac{n}{m}} &\text{ if } \sigma>\frac{m}{n}\\
\varepsilon^{-\frac{n}{m}}|\log\varepsilon| &\text{ if } \sigma=\frac{m}{n}\\
\varepsilon^{-\frac{1}{\sigma}} &\text{ if } 0<\sigma<\frac{m}{n}
\end{cases}
\end{equation}
for $\varepsilon\searrow 0$. Considering the approximation numbers $a_k(\id)$, this fits exactly to \eqref{1.2} now valid for $1\leq p<\infty$ as stated in Theorem \ref{Theo_ak}. In Section \ref{sec_ek} we transfer the last results to entropy numbers. We get in Theorem \ref{theo_ek}  for $k\in\nat,k\geq 2,$
\begin{equation}
e_k(\id)\sim
\begin{cases} 
k^{-\frac{m}{n}}\hspace{1cm}&  \text{ if } \sigma>\frac{m}{n}\\
{k}^{-\frac{m}{n}}\,(\log k)^{\frac{m}{n}}\hspace{1cm}&  \text{ if } \sigma=\frac{m}{n}\\
k^{-\sigma}\hspace{1cm}&  \text{ if } 0<\sigma<\frac{m}{n}.
\end{cases}
\end{equation}
Thereby \cite[Section 1.3.3]{ET96} provides upper bounds. Lower bounds can be constructed by using similar basis functions as in Section \ref{sec_ak}. Our main results are Theorem \ref{Theo_ak} and Theorem \ref{theo_ek}.

\section{Preliminaries}\label{sec_prelim}

We use standard notation. $\nat$ is the set of all natural numbers and $\no = \nat\cup\{0\}$, $\real$ is the set of all real numbers and $\rn,n\in\nat,$ is the Euclidean $n-$space. For two positive real sequences 
$\{\alpha_k\}_{k\in \nat}$, $\{\beta_k\}_{k\in \nat}$ or two positive functions $\phi(x)$, $\psi(x)$ we mean by 
 \[ \alpha_k\sim \beta_k \hspace{0.3cm}\text{or}\hspace{0.3cm}{\phi(x)\sim\psi(x)}\] that there exist constants $c_1,c_2>0$ such that for all $k\in\nat$ and all $x$ 
\[
c_1\ \alpha_k\leq \beta_k\leq c_2 \ \alpha_k\hspace{0.3cm}\text{or}\hspace{0.3cm}  c_1\ \phi (x)\leq\psi(x)\leq c_2\ \phi(x).
\] 
A complex-valued Lebesgue measurable function $u$, defined on $(0,\infty)$, belongs to $\Lloc([0,\infty))$ if  $u$ is Lebesgue integrable on each interval $(0,a)$ for $a>0.$ Let $\Omega$ be a smooth bounded domain in $\rn$. Recall that $\Com(\Omega)$ collects all complex-valued functions $f$ on $\rn$ having classical derivatives up to order $m\in\no$ with compact support $\supp f\subset\Omega$. Let $L_p(\Omega)$ with $1\leq p<\infty$ be the Lebesgue space of all complex-valued Lebesgue measurable functions on $\Omega$ such that
\begin{equation*}
\Vert f \vert L_p(\Omega)\Vert = \Big(\int_\Omega |f(x)|^p\dint x\Big)^{1/p}<\infty.
\end{equation*}
Furthermore, $D(\Omega)=C^\infty_0(\Omega)$ and the set of all complex distributions $D'(\Omega)$ have the usual meaning. $L_p(\Omega)$ and all other spaces introduced below are considered in the standard setting of $D'(\Omega)$.\\
Let $\overset{\circ\hspace{0.3cm} }{W^m_p}(\Omega)$ be the completion of $C^m_0(\Omega)$ with respect to the norm 
\begin{equation*}
\big\Vert f\, |\, \overset{\circ\hspace{0.3cm} }{W^m_p}(\Omega)\big\Vert 
= 
\left( \int_\Omega \sum_{|\alpha|=m}|\Dd^\alpha f(x)|^p
\dint x\right)^{1/p}.
\end{equation*}
For a linear and bounded operator $T\in\mathcal{L}(X,Y)$, acting between two complex quasi Banach spaces $X$ and $Y$, we define for $k\in\nat$ the \textit{$k$-th entropy number} $e_k(T)$ by 
\begin{equation*}
e_k(T):=\inf\big\{\varepsilon>0:\ \exists\, y_1,...y_{2^{k-1}}\in Y: T(B_X)\subseteq \bigcup\limits_{i=1}^{2^{k-1}} B_Y(y_i,\varepsilon)\big\}.
\end{equation*}
Here $B_X$ stands for the unit ball in $X$ and $B_Y(y_i,\varepsilon)$ is the set of all $y\in Y$ such that $\Vert y- y_i\Vert_Y<\varepsilon.$
The \textit{$k$-th approximation number} $a_k(T)$ is defined by
\begin{equation*}
a_k(T):=\inf\big\{\Vert T-S\Vert_{X\to Y}:\ S\in\mathcal{L}(X,Y), \text{ rank } S < k\big\}
\end{equation*} 
where $\text{rank } S = \dim S(X)$. We assume that the reader is familiar with entropy numbers, approximation numbers and their relations to the theory of function spaces. As for details, properties and historical comments, we refer the reader to \cite{CS90,EE87,ET96,HT08,Pi78,Pi87}.

Let $B=\{x\in\rn:  |x|<1\}$ be the unit ball in $\rn.$ We consider weighted Sobolev spaces on $B$ defined as follows.

\begin{defi}
Let $1\leq p<\infty$, $m\in\nat$ and $\sigma\geq 0$. Then $\Emp$ is the closure of $C_0^m(B)$ with respect to the norm 
\begin{equation}\label{def_norm}
\Vert f|\Emp\Vert = \bigg( \int_B |x|^{mp}(1+|\log|x||)^{\sigma p} \sum_{|\alpha|=m}|\Dd^\alpha f(x)|^p
\dint x\bigg)^{1/p}.
\end{equation}
\end{defi}

This Definition is justified by standard arguments. In fact, $E^m_{p,\sigma}(B)$ is the collection of all functions in $L_p(B)$ that are  limit elements of convergent sequences of functions in $C^m_0(B)$ in the norm \eqref{def_norm}. 
Furthermore, $\Emp$ is a Banach spaces (distributionally interpreted). We consider the continuous embedding
\begin{equation}
\id: \Emp\hookrightarrow L_p(B), \hspace{1cm} 1\leq p<\infty, m\in\nat, \sigma\geq 0.
\end{equation} 
For the justification of the continuity, we recall the (respectively rewritten) Hardy inequality \cite[Theorem 2.4]{Tri12}, or rather \cite[Corollary 1.8]{Mi15a}. Let $B^\delta=\{x\in\rn: |x|<\delta\},\,\delta>0.$

\begin{prop}\label{HaEq}
Let $m_1,m_2\in\no$ with $m_1\leq m_2$, $\sigma\geq 0$ and $1\leq p<\infty$. 
Then there are numbers $0<\delta<1$ and $c>0$ such that
\begin{align}\notag
\int_{\rn}& |x|^{m_1p}(1+|\log|x||)^{\sigma p}\sum_{|\alpha|=m_1}|\Dd^\alpha f(x)|^p\dint x\\\label{Heq_ad1}
&\leq
c\,\int_{\rn}|x|^{m_2p}(1+|\log|x||)^{\sigma p} \sum_{|\alpha|=m_2}|\Dd^\alpha f(x)|^p \dint x
\end{align}
for all $f\in C^{m_2}_0(B^\delta)$. In particular, if $m\in\nat$ then
\begin{equation}\label{Heq_ad2}
\int_{\rn}(1+|\log|x||)^{\sigma p}|f(x)|^p\dint x\ \leq\ c\,\int_{\rn} |x|^{mp}(1+|\log|x||)^{\sigma p}\sum_{|\alpha|=m}|\Dd^\alpha f(x)|^p \dint x
\end{equation}
for all $f\in C^{m}_0(B^\delta)$.
\end{prop}

In \cite[Theorem 3.3]{Tri12} it is shown that the embedding 
\[
\id: \Emp\hookrightarrow L_p(B) \text{ is compact if, and only if, }\sigma>0.
\]
We also refer to \cite[Proposition 2.5]{Mi15a} for details.

\section{Approximation numbers via \texorpdfstring{$L_p$}{Lp}-bracketing}\label{sec_ak}

We want to establish a bracketing method to estimate the asymptotic behaviour of the approximation numbers. Therefore we follow the idea of Evans and Harris \cite{EH93}, see also \cite[Chap. 6.3]{EE04}. 
We denote by 
\begin{equation}
\id_\Omega: E^m_{p,\sigma}(\Omega)\hookrightarrow L_p(\Omega)
\end{equation}
the restriction of $\id: \Emp\hookrightarrow L_p(B)$ to subsets $\Omega\subseteq B.$ We introduce the following quantities $\nu_0(\varepsilon, \Omega)$ and $\mu_0(\varepsilon, \Omega)$ taking over the notation from \cite{EH93}.

\begin{defi}\label{def_brack}
Let $1\leq p<\infty,m\in\nat,\sigma>0$ and $\varepsilon > 0$. We define 
\begin{equation}\label{def_nu}
\nu_0(\varepsilon, \Omega):= \max \{k \in\nat: a_k(\id_\Omega)\geq \varepsilon\}
\end{equation}
and put $\nu_0(\varepsilon, \Omega)= 0$ if $a_k<\varepsilon$ for all $k\in\nat.$ Furthermore, let
\begin{equation}\label{def_mu}
\mu_0(\varepsilon, \Omega)
:= 
\max\, \Big\{\,\dim S:\alpha(S):= \sup_{u\in S\setminus\{ 0\}}\dfrac{\Vert u| E^m_{p,\sigma}(\Omega)\Vert}{\Vert u| L_p(\Omega)\Vert}\leq \frac{1}{\varepsilon}\,\Big\} 
\end{equation}
where the maximum is taken over all finite-dimensional linear subspaces $S$ of $E^m_{p,\sigma}(\Omega).$
\end{defi}

\begin{rem} 
Due to the compactness of the embedding $\id$ we can assume that the approximation numbers tend to zero and hence the maximum in \eqref{def_nu} is attained for some natural number $N=N(\varepsilon).$ We will see in Proposition \ref{prop1} that this implies $\mu_0(\varepsilon,\Omega)<\infty$ for every $\varepsilon>0.$

The embedding $\id: E^m_{p,\sigma}(\Omega)\hookrightarrow L_p(\Omega)$ is injective. Thus for every finite-dimensional linear subspace $S\subset E^m_{p,\sigma}(\Omega)$ the restriction 
$
\id^S: S \to \id (S)
$
is bijective and bounded. We have
\begin{equation*}
\Vert (\id^S)^{-1}\Vert = \alpha(S)<\infty.
\end{equation*}
Clearly, 
\[
\nu_0(\varepsilon, \Omega)\to \infty
\hspace{1cm} \text{as } \varepsilon\to 0
\] 
describes the asymptotic behaviour of $a_k(\id_\Omega)\to 0$ as $k\to\infty.$ So the main concern of this section is to obtain upper and lower bounds for $\nu_0(\varepsilon, B)$.
\end{rem}

The essential tool is the following bracketing property.

\begin{prop}\label{prop1}
Let $\Omega=\Big(\bigcup\limits_{j=1}^J \overline{\Omega}_j\Big)^{\circ}$ with disjoint domains $\Omega_j$.\\ 
Then for $\varepsilon>0$ it holds
\begin{equation}\label{brack}
\sum\limits_{j=1}^J \mu_0(\varepsilon, \Omega_j)
\leq
\mu_0(\varepsilon, \Omega)
\leq
\nu_0(\varepsilon, \Omega)
\leq 
\sum\limits_{j=1}^J \nu_0(\varepsilon, \Omega_j).
\end{equation}
\end{prop}

\begin{proof}
\textbf{Step 1.} We prove the first inequality. Let $j\in\{1,...,J\}.$ We assume $\mu_0(\varepsilon,\Omega_j)<\infty$. Otherwise the assertion is trivial. 
 Then there exists a subspace $S_j\subset E^m_{p,\sigma}(\Omega_j)$ with 
$\dim S_j  = \mu_0(\varepsilon, \Omega_j)$ such that
\[
\Vert u| E^m_{p,\sigma}(\Omega_j)\Vert \leq \frac{1}{\varepsilon}\,\Vert u|L_p(\Omega_j)\Vert,\ u\in S_j.
\]
We put $S:= \bigoplus\limits_{j=1}^J S_j\subseteq E^m_{p,\sigma}(\Omega).$ For $v\in S$, say $v=\sum\limits_{j=1}^J u_j, u_j\in S_j$, we get according to the disjointness of the domains $\Omega_j$
\begin{equation*}
\Vert v | E^m_{p,\sigma}(\Omega)\Vert^p
=
\sum\limits_{j=1}^J \Vert  u_j  | E^m_{p,\sigma}(\Omega_j)\Vert^p
\leq
\sum\limits_{j=1}^J \frac{1}{ \varepsilon^p} \Vert  u_j | L_p(\Omega_j)\Vert^p
= \frac{1}{ \varepsilon^p} \Vert  v | L_p(\Omega)\Vert^p.
\end{equation*}
Thus $S$ is an admitted subspace in \eqref{def_mu} and we conclude 
\[
\mu_0(\varepsilon,\Omega) \geq \dim S = \sum_{j=1}^J \mu_0(\varepsilon, \Omega_j).
\]
\textbf{Step 2.} We prove the second inequality.
Let $S$ be a finite-dimensional linear subspace of $E^m_{p,\sigma}(\Omega)$ and $P:E^m_{p,\sigma}(\Omega)\to L_p(\Omega)$ be a finite rank operator with 
\[
\rank P<\dim S=:d.
\]
Then there is an element $0\neq f^*\in S$ with $P(f^*)=0$. Denote $f^*=\sum\limits_{i=1}^d\lambda_i e_i$ where $\sum\limits_{i=1}^d|\lambda_i|\neq 0$ and $S=\spa\{e_1,...,e_d\}$. Then
\[
\Vert (\id - P) f^*\ |\, L_p(\Omega)\Vert 
= 
\Vert f^*\ |\, L_p(\Omega)\Vert
\geq
\alpha(S)^{-1} \Vert f^*\ |\, E^m_{p,\sigma}(\Omega)\Vert
\]
and hence
\[
\Vert \id - P \Vert \geq \alpha(S)^{-1}. 
\]
At this point we have seen that 
\[
a_d(\id)\geq \alpha(S)^{-1}
\]
for all finite dimensional subspaces $S\subset E^m_{p,\sigma}(\Omega)$  with $\dim S=d.$ This means 
\[
a_d(\id)\geq \varepsilon
\]
for all finite dimensional subspaces $S\subset E^m_{p,\sigma}(\Omega)$ with $\dim S=d$ and in addition $\alpha(S)\leq\tfrac{1}{\varepsilon}.$
Hence,
\[
\nu_0(\varepsilon,\Omega)\geq \dim S
\]
for all finite dimensional subspaces $S\subset E^m_{p,\sigma}(\Omega)$ with $\alpha(S)\leq\tfrac{1}{\varepsilon}.$ This finishes the verification of $\nu_0(\varepsilon,\Omega)\geq\mu_0(\varepsilon,\Omega)$.\\
\textbf{Step 3.} We prove the last inequality.
For $k:=\nu_0(\varepsilon, \Omega_j) + 1$ we have
\[
a_k(\id^j:E^m_{p,\sigma}(\Omega_j)\hookrightarrow L_p(\Omega_j)) < \varepsilon.
\]
In other words, for every $j=1,...,J$ there exists a linear and bounded operator $P_j: E^m_{p,\sigma}(\Omega_j)\to L_p(\Omega_j)$ with $\rank P_j\leq \nu_0(\varepsilon, \Omega_j)$ such that 
\[
\Vert \id^j-P_j| E^m_{p,\sigma}(\Omega_j)\to L_p(\Omega_j)\Vert < \varepsilon.
\]
Let $P$ be the operator defined by
\[
(Pf)(x):= \sum_{j=1}^J \chi_{\Omega_j}(x)\ (P_jf)(x),\hspace{0.5cm} f\in E^m_{p,\sigma}(\Omega).
\]
Then it holds for all $f\in E^m_{p,\sigma}(\Omega)$
\begin{align*}
\Vert f - Pf | L_p(\Omega)\Vert ^p 
=& 
\int_\Omega |f(x) - \sum_{j=1}^J \chi_{\Omega_j}(x)\ P_jf(x)|^p\dint x\\
=&
\sum_{j=1}^J \int_{\Omega_j} |f(x) - P_jf(x)|^p\dint x\\
\leq&
\sum_{j=1}^J \Vert \id^j - P_j\Vert^p\ \Vert f| E^m_{p,\sigma}(\Omega_j)\Vert^p\\
<&
\varepsilon^p\ \Vert f| E^m_{p,\sigma}(\Omega)\Vert^p.
\end{align*}
Thus for $L:=1+\sum_{j=1}^J \nu_0(\varepsilon, \Omega_j)$, we have $a_L(\id)< \varepsilon$.
Hence,
\[
\nu_0(\varepsilon, \Omega)= \max \{l: a_l(\id)\geq \varepsilon\}\ \leq\ L -1 = \sum_{j=1}^J \nu_0(\varepsilon, \Omega_j).
\]
\end{proof}

The last Proposition can be seen as an $L_p$-version of the Dirichlet-Neumann bracketing method which is so effective in spectral $L_2$-theory. Indeed Triebel used this technique in \cite{Tri12} to obtain eigenvalue distribution of the degenerate elliptic operator defined by
\begin{align}
\label{op1}
&A^m_\sigma f = (-1)^m\sum_{|\alpha|=m} \Dd^\alpha(b_{m,\sigma}\Dd^\alpha f)
\\ 
\label{op2}
&\dom((A^m_\sigma)^{\frac{1}{2}}) = E^m_{2,\sigma}(B)
\end{align}
where $b_{m,\sigma}(x)=|x|^{2m}(1+|\log|x||)^{2\sigma}, x\in B$. These spectral results then provided estimates for the approximation numbers $a_k(\id).$ 
We will state the connection between this $L_2$-bracketing method and the introduced $L_p$-bracketing quantities $\nu_0(\varepsilon,B)$ and $\mu_0(\varepsilon,B)$.

\begin{prop}\label{p=2}
For $p=2$ let 
 \[
 \id:E^m_{2,\sigma}(B)\hookrightarrow L_2(B)
 \]
be the compact embedding. Then for every $\varepsilon>0$ it holds
\begin{equation}\label{nu=mu}
\nu_0(\varepsilon,B)=\mu_0(\varepsilon,B). 
\end{equation}
Furthermore, the positive definite, self-adjoint degenerate elliptic operator $A^m_\sigma$ has pure point spectrum $(\lambda_k(A^m_\sigma))_{k=1}^\infty$, monotonically ordered including geometric multiplicities. It holds
 \begin{equation}\label{nu=N}
 \nu_0(\varepsilon,B)= N(\varepsilon^{-2},A^m_\sigma)
 \end{equation}
 where $N(\lambda,A^m_\sigma)$ denotes the number of eigenvalues $\lambda_k(A^m_\sigma)$ smaller than or equal to $\lambda>0.$
\end{prop}

\begin{proof}
Let $\id^*: L_2(B)\hookrightarrow E^m_{2,\sigma}(B)$ be the dual map of $\id$ defined by
\begin{equation}\label{dual}
\big( \id f, g\, \big)_{L_2(B)} 
=
\big( f, \id^* g\,\big)_{E^m_{2,\sigma}(B)}
\hspace{1cm}
\forall f\in E^m_{2,\sigma}(B), g\in L_2(B). 
\end{equation}
                                                                                                                                                                                                                                                                                                                                                                                                                                                                                                                                                      Here $(\cdot,\cdot)_{L_2(B)}$ denotes the inner product in $L_2(B)$. Respectively the inner product $(\cdot,\cdot)_{E^m_{2,\sigma}(B)}$ in $E^m_{2,\sigma}(B)$ is given by the closure of the quadratic form
\begin{equation}\label{form}
\int_B b_{m,\sigma}^2(x)\sum_{|\alpha|=m}\overline{\Dd^\alpha f(x)}\Dd^\alpha g(x)\dint x, \hspace{1cm} f,g\in C^m_0(B).
\end{equation}
The closure $(\cdot,\cdot)_{E^m_{2,\sigma}(B)}$ of \eqref{form} generates  the positive definite, self-adjoint operator $A^m_\sigma$ given by \eqref{op1} and \eqref{op2}. That means
\begin{equation*}
\big( A^m_\sigma f, \id g\big)_{L_2(B)}
= 
\big( f, g\big)_{E^m_{2,\sigma}(B)}
\hspace{1cm}
f\in\dom A^m_\sigma,g\in E^m_{2,\sigma}(B).
\end{equation*} 
The embedding $\id:E^m_{2,\sigma}(B)\hookrightarrow L_2(B)$ is compact. Hence, by Rellich's Criterion \cite[Section 4.5.3, p. 258]{Tri92} the operator $A^m_\sigma$ has pure point spectrum. Moreover, the approximation numbers coincide with its singular values, see for instance \cite[Theorem II.5.10, p.91]{EE87}. That is
\begin{equation}\label{sing}
a_k(\id) = \lambda_k (|\id|) = \lambda_k([\,\id^*\circ\id\,]^{1/2})
\end{equation}
where $\lambda_k(\cdot)$ denotes the $k$th-eigenvalues of the corresponding operator.
Furthermore 
\[
\id^*\circ\id: E^m_{2,\sigma}(B)\to E^m_{2,\sigma}(B)
\]
is a non-negative, compact and selfadjoint operator. Respectively we apply \cite[Theorem II.5.6, p.84]{EE87} to $T=\id^*\circ\id$. We get that
\begin{equation*}
\#\{k: \lambda_k(T)\geq \varepsilon^2\}
= \max \dim S 
\end{equation*}
where the maximum is taken over all closed linear subspaces $S$ of $E^m_{2,\sigma}(B)$ such that for all $f\in S$
\begin{equation*}
\big( Tf, f \big)_{E^m_{2,\sigma}(B)} \geq \varepsilon^2 \Vert f| E^m_{2,\sigma}(B)\Vert^2.
\end{equation*}
Due to \eqref{dual}, last line is equivalent to 
\begin{equation*}
\alpha(S) = \sup_{f\in S, f\neq 0} \frac{\Vert f| E^m_{2,\sigma}(B)\Vert}{\Vert f| L_2(B)\Vert}\leq \frac{1}{\varepsilon}.
\end{equation*}
We have shown
\[
\nu_0(\varepsilon,B) = \#\{k: a_k(\id)\geq \varepsilon\} = \#\{k: \lambda_k(T)\geq \varepsilon^2\} \leq \mu_0(\varepsilon,B).
\]
The converse inequality was already shown in \eqref{brack}. This proves \eqref{nu=mu}. Next we justify \eqref{nu=N}. Recall  that $\id$ and $\id^*$ have the same singular values, see \cite[Theorem II.5.7, p. 85]{EE87}. Hence, in view of \eqref{sing} it suffices to prove that  
\begin{equation}\label{invers}
\id\circ \id^* = (A^m_\sigma)^{-1}
\end{equation}
where we consider the operators from $L_2(B)$ to $L_2(B)$. Therefore let $f\in\dom A^m_\sigma$ and $g\in E^m_{2,\sigma}(B).$ Then
\[
\big( \id^* (A^m_\sigma f), g\big)_{E^m_{2,\sigma}(B)}
=
\big( A^m_\sigma f, \id g\big)_{L_2(B)}
= 
\big( f, g\big)_{E^m_{2,\sigma}(B)}.
\]
Hence for all $f\in \dom A^m_\sigma$
\begin{equation*}
\id^* (A^m_\sigma f) = f.
\end{equation*}
The inverse operator $(A^m_\sigma)^{-1}$ acting in $L_2(B)$ is given by
\begin{align*}
\dom (A^m_\sigma)^{-1}&:=\{g\in L_2(B):\,\exists\,f\in\dom A^m_\sigma, A^m_\sigma f = g\}\\
(A^m_\sigma)^{-1} g& := \id f.
\end{align*}
Now \eqref{invers} follows  by
\[
(A^m_\sigma)^{-1} g = \id f = \id ( \id^* A^m_\sigma f) = \id (\id^* g).
\]
Finally
\begin{align*}
\nu_0(\varepsilon,B) 
&= \#\{k: a_k(\id)\geq \varepsilon\} 
= \#\{k: \lambda_k(\id^*\circ\id)\geq \varepsilon^2\}\\
&= \#\{k: \lambda_k((A^m_\sigma)^{-1})\geq \varepsilon^2\}
= \#\{k: \lambda_k(A^m_\sigma)\leq \varepsilon^{-2}\}.
\end{align*}
\end{proof}

In the sense of the last Proposition, the bracketing quantities $\nu_0(\varepsilon,B)$ and $\mu_0(\varepsilon,B)$ extend the Courant-Weyl method of Dirichlet-Neumann bracketing. For details of this method in Hilbert spaces, we refer to\cite[Chapter XI]{EE87}.

We turn to the situation of Banach spaces, $1\leq p<\infty$, and concentrate on the asymptotic behaviour of $\nu_0(\varepsilon,B)$ as $\varepsilon\to 0.$ First we will show that one can cut off the singularity at $x=0$ without affecting the asymptotic behaviour of $\nu_0(\varepsilon,B)$ and $\mu_0(\varepsilon,B)$. This is the crucial point to control the singularity in our situation.

\begin{prop}\label{prop2}
Let $\varepsilon>0$ and put $B_J:=\{x\in\rn: |x|<2^{-J}\}$ with $J=J(\varepsilon)\in\nat$ such that $J\sim\varepsilon^{-\frac{1}{\sigma}}$. Then it holds
\begin{equation}\label{nu=mu=0}
\nu_0(\varepsilon,B_J)=\mu_0(\varepsilon,B_J) = 0.
\end{equation}
\end{prop}

\begin{proof}
Due to \eqref{brack} it suffices to prove  that $\nu_0(\varepsilon,B_J)=0$. 
Let $0<\delta<1$ be the constant from Proposition \ref{HaEq}. Without loss of generality we assume $2^{-J}\leq \delta$. Then for $f\in E^m_{p,\sigma}(B_J)$
\[
\Vert f|L_p(B_J)\Vert^p 
\leq 
c\, J^{-\sigma p}\int_{B_J}(1+|\log|x||)^{\sigma p}|f(x)|^p\dint x
\leq 
c\, J^{-\sigma p} \Vert f|E^m_{p,\sigma}(B_J)\Vert^p.
\]
With $J\sim\varepsilon^{-\frac{1}{\sigma}}$ one has
\[
\Vert \id_J: E^m_{p,\sigma}(B_J)\hookrightarrow L_p(B_J)\Vert <\varepsilon.
\]
This proves the assertion since the approximation numbers $a_k(\id_J)$ are bounded by the norm $\Vert \id_J\Vert.$
\end{proof}

With the last bracketing results we are able to adapt the Dirichlet-Neumann bracketing in $L_2(B)$ used in \cite{Tri12} to our situation. We decompose the domain $B$ into finitely many annuli leaving a small ball around the origin and consider restricted operators separately. We get rid of the singularity due to the last Proposition. The quantities $\nu_0(\varepsilon,B^j)$ and $\mu_0(\varepsilon,B^j)$ restricted to annuli $B^j$ deliver lower and upper bounds. The result reads as follows.

\begin{prop}\label{nu_0}
Let $n,m\in\nat$ and $\varepsilon>0$ small. Then it holds
\begin{equation}\label{n0}
\nu_0(\varepsilon, B)
\sim
\begin{cases} 
\varepsilon^{-\frac{n}{m}} &\text{ if } \sigma>\frac{m}{n}\\
\varepsilon^{-\frac{n}{m}}|\log\varepsilon| &\text{ if } \sigma=\frac{m}{n}\\
\varepsilon^{-\frac{1}{\sigma}} &\text{ if } 0<\sigma<\frac{m}{n}.
\end{cases}
\end{equation}
\end{prop}

\begin{proof}
 Let $J\in\nat$ with $J\sim \varepsilon^{-\frac{1}{\sigma}}$ and 
 denote $B^j:=\{x\in B:\  2^{-j}\leq|x|<2^{-j+1}\}$, $j=1,...,J$ and $B_J:=\{x\in B:\ |x|<2^{-J}\}$.
 Consider the disjoint partition of the unit ball
\begin{equation*}
B = B_J\ \cup\ (B \setminus B_J) 
\end{equation*} 
where  $B \setminus B_J := \bigcup\limits_{j=1}^{J}B^j$.
By \eqref{brack} and \eqref{nu=mu=0}, one has 
\begin{align}\notag 
\mu_0(\varepsilon, B \setminus B_J)\
&=
\mu_0(\varepsilon, B \setminus B_J) + \mu_0(\varepsilon, B_J)\\\notag 
&\leq
\nu_0(\varepsilon, B)\\\notag 
&\leq
\nu_0(\varepsilon, B \setminus B_J) + \nu_0(\varepsilon,B_J)\\\label{n0B}
&=
\nu_0(\varepsilon, B \setminus B_J)
\end{align}
\textbf{Step 1.} 
We prove the upper bounds for \eqref{n0}. On the annuli $B^j$ one can replace the weights by proportional constants. Hence
\[
\Vert \id^j| E^m_{p,\sigma}(B^j)\hookrightarrow \overset{\circ\hspace{0.3cm}}{W^m_p}(B^j)\Vert 
\leq 
c\, 2^{jm}j^{-\sigma}.
\]
Furthermore, recall the well-known classical result for smooth bounded domains $\Omega\subset\rn$ 
\begin{equation}\label{res_class}
a_k(\overset{\circ\hspace{0.3cm}}{W^m_p}(\Omega)\hookrightarrow L_p(\Omega))\sim k^{-\frac{m}{n}}.
\end{equation}
For the history of these results we refer to \cite[Section 3.3.5]{ET96}.
It follows by the same dilation arguments as in \cite{Tri12} that
\begin{equation*}
a_k(\overset{\circ\hspace{0.3cm}}{W^m_p}(B^j)\hookrightarrow L_p(B^j))\sim 2^{-jm} k^{-\frac{m}{n}}.
\end{equation*}
By decomposition of  $\id^j:E^m_{p,\sigma}(B^j)\hookrightarrow L_p(B^j)$ we get 
\[
a_k(\id^j)
\leq
c\, \Vert \id^j| E^m_{p,\sigma}(B^j)\hookrightarrow \overset{\circ\hspace{0.3cm}}{W^m_p}(B^j)\Vert\, a_k(\overset{\circ\hspace{0.3cm}}{W^m_p}(B^j)\hookrightarrow L_p(B^j))
\leq 
c\, j^{-\sigma} k^{-\frac{m}{n}}
\]
and thus $\nu_0(\varepsilon, B^j) \leq c\, j^{-\frac{n}{m}\sigma}\varepsilon^{-\frac{n}{m}}.$ The constants are independent of $j$. Then one has by \eqref{brack}
\begin{align*}
\nu_0&(\varepsilon, B \setminus B_J)\
 \leq
\sum\limits_{j=1}^{J} \nu_0(\varepsilon, B^j)\\
 &\leq
c\ \varepsilon^{-\frac{n}{m}}
 \sum\limits_{j=1}^{J} j^{-\frac{n}{m}\sigma} 
 \sim \varepsilon^{-\frac{n}{m}}
\begin{cases} 
1  &\text{ if } \sigma>\frac{m}{n}\\
\log J  &\text{ if } \sigma=\frac{m}{n}\\
J^{1-\frac{n}{m}\sigma}  &\text{ if } 0<\sigma<\frac{m}{n}.
\end{cases}
\end{align*}
The estimates from above in \eqref{n0} follow now from $ J\sim \varepsilon^{-\frac{1}{\sigma}}$ and \eqref{n0B}.\\
\textbf{Step 2.} We prove the estimates from below in \eqref{n0}. Therefore we shall construct suitable finite-dimensional subspaces to estimate $\mu_0(\varepsilon, B \setminus B_J)$ from below. We use basis functions similar to those in \cite{HT94}, see also \cite[Section 4.3.2, p.170-173, Step 1-2]{ET96}. Let $f\in S'(\rn)$ with $\supp f\subseteq [-1,1]^n.$ Put
\begin{equation*}
 \spa^l_j:=\spa\{f(2^l\cdot-k):k\in\zn, 2^{-l}k\in B^j\} \hspace{1cm} j,l\in\nat,l\geq j.
\end{equation*}
The number of admitted lattice points $k\in\zn$ such that $2^{-l}k\in B^j$ is $2^{n(l-j)}$ (neglecting constants). Furthermore one may assume that the functions $f(2^{l}\cdot -k)$ have disjoint supports and therefore
\begin{equation*}
\dim \spa^l_j \sim 2^{n(l-j)}.
\end{equation*}
For every $g\in\spa^l_j$, say
\begin{equation*}
g(x)= \sum\nolimits^{l,j} a_k f(2^lx-k), \hspace{1cm}a_k\in\mathbb{C},\, j,l\in\nat, l\geq j,
\end{equation*}
where the sum is taken over all lattice points $k\in\zn$ such that $2^{-l}k\in B^j$, it holds 
\begin{equation*}
\Vert g\,|L_p(B)\Vert^p 
\sim  
\sum\nolimits^{l,j} |a_k|^p\, \Vert f(2^l\cdot-k)|L_p(B)\Vert^p\ \sim 2^{-ln} \sum\nolimits^{l,j} |a_k|^p.
\end{equation*}
Then $\supp f(2^l\cdot-k)\subset B^j$ leads to
\begin{equation*}
\Vert g\,|E^m_{p,\sigma}(B)\Vert^p 
\sim  
\sum\nolimits^{l,j} |a_k|^p\, \Vert f(2^l\cdot-k)|E^m_{p,\sigma}(B)\Vert^p\ \sim j^{\sigma p} 2^{m(l-j)p}2^{-ln} \sum\nolimits^{l,j} |a_k|^p.
\end{equation*}
In particular,
\begin{equation}\label{Norm_flj}
\Vert g\,|E^m_{p,\sigma}(B)\Vert\ \sim\ j^\sigma 2^{m(l-j)}\ \Vert g\,|L_p(B)\Vert\hspace{1cm}\forall\, g\in\spa^l_j.
\end{equation}
Note that one can replace $B$ by $B^j$ in \eqref{Norm_flj} due to the construction of $\spa^l_j.$ We deal with three different subspaces to obtain the three estimates from below in \eqref{n0}. Firstly, let $L\in\nat$ such that $L\sim -\frac{1}{m}\log \varepsilon$ and 
\[
S_1:= \spa^L_1.
\]
Then for every $g\in S_1$ we have due to \eqref{Norm_flj}
\[
\Vert g\,|E^m_{p,\sigma}(B)\Vert\ \sim\ 2^{mL}\ \Vert g\,|L_p(B)\Vert.
\]
Hence $\alpha(S_1)\leq 2^{mL}\sim\frac{1}{\varepsilon}.$ This ensures
\[
\mu_0(\varepsilon,B \setminus B_J)\geq \dim S_1 \sim 2^{nL}\sim \varepsilon^{-\frac{n}{m}}.
\]
The second subspace is defined by
\[
S_2:= \bigoplus_{j=1}^J \spa^j_j.
\]
Then for every $g\in S_2$, say $g=\sum\limits_{j=1}^J g_j$ with $g_j\in\spa^j_j,$ we get with \eqref{Norm_flj} 
\[
\Vert g\,|E^m_{p,\sigma}(B)\Vert^p\ 
\sim\ 
\sum_{j=1}^J \Vert g_j\,|E^m_{p,\sigma}(B^j)\Vert^p
\sim\
\sum_{j=1}^J  j^{\sigma p}\Vert g_j\,|L_p(B^j)\Vert^p.
\]
We estimate $j^\sigma$ by $J^{\sigma}$ and obtain $\alpha(S_2)\leq J^\sigma\sim\frac{1}{\varepsilon}.$ Consequently we have
\[
\mu_0(\varepsilon,B \setminus B_J)\geq \dim S_2 \sim J \sim \varepsilon^{-\frac{1}{\sigma}}.
\]
Up to now we have shown
\[
\mu_0(\varepsilon,B \setminus B_J)\geq c\, \varepsilon^{-\max\{\frac{1}{\sigma},\frac{n}{m}\}}.
\]
In the limiting case $\sigma=\frac{m}{n}$ we can refine the decomposition of $B \setminus B_J$ to obtain the $\log$-factor. Define a subspace
\[
S_3:= \bigoplus_{j=1}^J \spa^{l_j}_j
\]
where $l_j\sim j+\frac{1}{n}(\log J -\log j).$ Then for every $g\in S_3$, say $g=\sum\limits_{j=1}^J g_j$ with $g_j\in \spa^{l_j}_j$, it holds because of \eqref{Norm_flj}
\begin{align*}
\Vert g| E^m_{p,\sigma}(B)\Vert^p
&\sim\
\sum\limits_{j=1}^J  \Vert g_j| E^m_{p,\sigma}(B^j)\Vert^p \\
&\sim\
\sum\limits_{j=1}^J j^{\sigma p}\, 2^{m(l_j-j)p}\,\Vert g_j| L_p(B^j)\Vert^p\\
&\sim\
J^{\frac{m}{n}p} \sum\limits_{j=1}^J j^{(\sigma-\frac{m}{n})p}\, \Vert g_j| L_p(B^j)\Vert^p.
\end{align*}
If $\sigma=\frac{m}{n}$ then $j^{(\sigma-\frac{m}{n})p}=1$ and so 
 $\alpha(S_3)\leq J^{\frac{m}{n}}\sim\frac{1}{\varepsilon}$. We conclude for $\sigma=\frac{m}{n}$
\[
\mu_0(\varepsilon,B \setminus B_J)\geq \dim S_3\sim \sum\limits_{j=1}^J 2^{n(l_j-j)}\sim J\sum\limits_{j=1}^J \frac{1}{j} \sim J\log J \sim \varepsilon^{-\frac{n}{m}}|\log \varepsilon|.
\]
The estimates from below in \eqref{n0} follow now from \eqref{n0B}.
\end{proof}

We will transfer the asymptotic behaviour of $\nu_0(\varepsilon,B)$ as $\varepsilon\to 0$ to the corresponding approximation numbers.

\begin{theo}\label{Theo_ak}
Let $n,m\in\nat, 1\leq p<\infty$ and $\sigma>0$. Then the embedding $\id: \Emp\hookrightarrow L_p(B)$ is compact and 
\begin{equation}\label{ak}
a_k(\id)\sim
\begin{cases} 
k^{-\frac{m}{n}}\hspace{1cm}&  \text{ if } \sigma>\frac{m}{n}\\
{k}^{-\frac{m}{n}}\,(\log k)^{\frac{m}{n}}\hspace{1cm}&  \text{ if } \sigma=\frac{m}{n}\\
k^{-\sigma}\hspace{1cm}&  \text{ if } 0<\sigma<\frac{m}{n}.
\end{cases}
\end{equation}
\end{theo}

\begin{proof}
First we remark that from $\nu_0(\varepsilon,B)\sim \varepsilon^{-\kappa}|\log\varepsilon|^{\rho}$, $\kappa>0,\rho\in\real,$ it follows $\nu_0(a_k(\id),B)\sim k$. Hence
\begin{equation}\label{ksim}
k\sim a_k(\id)^{-\kappa}|\log a_k(\id)|^{\rho}.
\end{equation}
Then one has in particular
\begin{equation*}
\log k 
\sim 
|\log a_k(\id)|\Big[\kappa +\rho \frac{\log|\log a_k(\id)|}{|\log a_k(\id)|} \Big]
\sim
|\log a_k(\id)|.
\end{equation*}
Inserting this in \eqref{ksim} one obtains for $k\geq 2$
\[
a_k(\id)\sim k^{-\frac{1}{\kappa}} (\log k)^{\frac{\rho}{\kappa}}.
\]
Now \eqref{ak} follows from \eqref{n0}.
\end{proof}

\section{Entropy numbers}\label{sec_ek}

We turn to the entropy numbers $e_k(\id)$ of the compact embedding $\id:\Emp\hookrightarrow L_p(B).$ We have so far the exact behaviour of the approximation numbers in Theorem \ref{Theo_ak}. The corresponding upper bounds for the entropy numbers follow immediately from \cite[Section 1.3.3]{ET96}. On the other hand, similar constructions to those from the proof of Theorem \ref{Theo_ak} lead to the estimates from below of $e_k(\id)$. The result reads as follows.

\begin{theo}\label{theo_ek}
Let $n,m\in\nat, 1\leq p<\infty$ and $\sigma>0$. Then the embedding $\id: \Emp\hookrightarrow L_p(B)$ is compact and 
\begin{equation}\label{ek}
e_k(\id)\sim
\begin{cases} 
k^{-\frac{m}{n}}\hspace{1cm}&  \text{ if } \sigma>\frac{m}{n}\\
{k}^{-\frac{m}{n}}\,(\log k)^{\frac{m}{n}}\hspace{1cm}&  \text{ if } \sigma=\frac{m}{n}\\
k^{-\sigma}\hspace{1cm}&  \text{ if } 0<\sigma<\frac{m}{n}.
\end{cases}
\end{equation}
\end{theo}

\begin{proof}
\textbf{Step 1.} 
Considering \eqref{ak}, one has in particular that
\[
a_{2^{j-1}}(\id)\sim a_{2^j}(\id),\hspace{1cm}j\in\nat.
\]
Now one can apply \cite[Section 1.3.3, p.15]{ET96} with reference to \cite{Tri94} and gets
\[
e_k(\id)\leq c\, a_k(\id),\hspace{1cm}k\in\nat.
\]
\textbf{Step 2.} 
By decomposition of $\id$ and \eqref{res_class}, we clearly get due to the multiplicativity of entropy numbers the classical lower estimate 
\[
k^{-\frac{m}{n}}
\sim e_k(\overset{\circ\hspace{0.3cm} }{W^m_p}(B)\hookrightarrow L_p(B))
\leq c\, e_k(\id).
\]
We claim
\[
e_k(\id)\geq c\ k^{-\sigma}.
\]
We adapt arguments from \cite{HT94} or rather \cite[Theorem 4.3.2, Step 1]{ET96}. Let 
\begin{equation}\label{flj2}
f^l_j(x):= \sum\nolimits^{l,j} a_k f(2^lx-k), \hspace{1cm}a_k\in\mathbb{C},\, j,l\in\nat, l\geq j,
\end{equation}
where $f\in S'(\rn)$ is such that $\supp f\subset[-1,1]^n$. The sum $\sum\nolimits^{l,j}$ is taken over all lattice points $k\in\zn$ such that $2^{-l}k\in B^j$. The number of the subcommands is $N_{l-j}= 2^{n(l-j)}$ (neglecting constants). As before in the proof of Proposition \ref{nu_0}, we assume that $f(2^l\cdot-k)$ have disjoint supports. We obtain
\begin{equation}\label{1}
\Vert f^l_j\,|L_p(B)\Vert
\sim  
2^{-l\frac{n}{p}} \big(\sum\nolimits^{l,j} |a_k|^p\big)^{\frac{1}{p}}
\end{equation}
and
\begin{equation}\label{2}
\Vert f^l_j\,|E^m_{p,\sigma}(B)\Vert  \sim j^{\sigma} 2^{m(l-j)}2^{-l\frac{n}{p}} \big(\sum\nolimits^{l,j} |a_k|^p\big)^{\frac{1}{p}}.
\end{equation}
Due to the definition of entropy numbers, there exist $2^{N_{l-j}}$ balls $K^i$, $i=1,...,2^{N_{l-j}}$ in $L_p(B)$ with radius $\tilde{\varepsilon}=2e_{N_{l-j}}(\id)$ which cover the unit ball $U$ of $\Emp.$ For any one of these balls $K=K^i$ it holds
\[
\vol(K\cap\spa^l_j)\leq c\big[2 e_{N_{l-j}}(\id) 2^{l\frac{n}{p}}\big]^{N_{l-j}}\vol(U^{N_{l-j}}_p)
\]
where $U^N_p$ is the unit ball in $\ell_p^N$. We can estimate
\begin{align}\notag
\vol(U\cap\spa^l_j)
&\leq 
\sum_{i=1}^{2^{N_{l-j}}} \vol(K^i\cap\,\spa^l_j)\\ \label{3}
&\leq
c\, 2^{N_{l-j}}\big[2\,e_{N_{l-j}}(\id)\,2^{l\frac{n}{p}}\,\big]^{N_{l-j}} \vol(U^{N_{l-j}}_p).
\end{align}
The left hand side is equivalent to $[j^{-\sigma}2^{-m(l-j)}2^{l\frac{n}{p}}]^{N_{l-j}}\vol(U^{N_{l-j}}_p)$ and so we have
\begin{equation}\label{1.3}
j^{-\sigma}2^{-m(l-j)} \leq c\, e_{N_{l-j}}(\id).
\end{equation}
If $\sigma=\frac{m}{n}$ we choose $j=1.$ Otherwise let $l$ and $j$ be such that 
\begin{align*}
l\sim j +\tfrac{\sigma}{n\sigma - m}\log j 
&\ \Longleftrightarrow\
(n-\tfrac{m}{\sigma})(l-j) \sim \log j\\
&\ \Longleftrightarrow\
2^{n(l-j)}2^{-\tfrac{m}{\sigma}(l-j)} \sim j\\
&\ \Longleftrightarrow\
2^{n(l-j)} \sim j\, 2^{\tfrac{m}{\sigma}(l-j)}\\
&\ \Longleftrightarrow\
[N_{l-j}]^{-\sigma} \sim j^{-\sigma}\, 2^{-m(l-j)}.
\end{align*}
Then \eqref{1.3} leads to
\[
e_k(\id)\geq c\, k^{-\sigma}.
\]
\textbf{Step 3.} We prove the limiting case $\sigma=\frac{m}{n}.$ We  fix $J\in\nat$ and construct in each annulus $B^j, j=1,...,J$ functions of type \eqref{flj2} such that the size of the lattice depends on $j$. Namely, consider
\[
f^J(x):= \sum_{j=1}^J f^{l_j}_j(x),\hspace{1cm} b_j\in\mathbb{C},
\]
such that $l_j\sim j + \tfrac{1}{n}(\log J - \log j)$. Let  
\[
\spa^J:=\spa\big\{f(2^{-l_j}x -\, k): k\in\zn, 2^{-l_j}k\in B^j,\, j=1,...,J\big\}.
\] 
Then
\[
\dim\spa^J 
\sim\, 
\sum_{j=1}^J 2^{n(l_j-j)}
\sim\, 
J\sum_{j=1}^J\frac{1}{j} 
\sim\,
J\log J. 
\]
We have the following counterparts of \eqref{1} and \eqref{2} with modified coefficients $b_k=2^{-l_j\frac{n}{p}}a_k$
\[
\Vert f^J| L_p(B)\Vert 
\sim 
\sum_{j=1}^J \Vert f^{l_j}_j| L_p(B)\Vert
\sim
\big(\sum\nolimits^\ast |b_k|^p \big)^{\frac{1}{p}}.
\]
Since $ J^{\frac{m}{n}}\sim j^{\sigma}2^{m(l_j-j)}$ if $\sigma=\frac{m}{n}$, we obtain
\begin{align*}
\Vert f^J| E^m_{p,\sigma}(B)\Vert
&\sim
\sum_{j=1}^J \Vert f^{l_j}_j| E^m_{p,\sigma}(B)\Vert
\sim
\sum_{j=1}^J j^{\sigma}2^{m(l_j-j)}\Vert f^{l_j}_j| L_p(B)\Vert\\
&\sim
J^{\frac{m}{n}}\big(\sum\nolimits^\ast |b_k|^p \big)^{\frac{1}{p}}.
\end{align*}
The sum $\sum\nolimits^\ast$ is taken over $N^J\sim J\log J$ summands. Now we are in the same situation as in Step 2. For $2^{N^J}$ balls $(K^i)$, $i=1,...,2^{N^J},$ with radius $\tilde{\varepsilon}=2 e_{N^J}(\id)$, which cover the unit ball $U$, it holds as a counterpart of \eqref{3}
\[
\vol(U\cap\spa^J)
\leq 
\sum_{i=1}^{2^{N^J}}\vol(K^i\cap\spa^J)
\leq
c\, 2^{N^J} \big[2e_{N^J}(\id)\big]^{N^J}\vol(U^{N^J}_p).
\]
Similarly as in Step 2, the left hand side is equivalent to $[J^{-\frac{m}{n}}]^{N^J}\vol(U^{N^J}_p).$ Hence, we showed that
\[
J^{-\frac{m}{n}}\leq c\, e_{N^J}(\id).
\]
Finally,
\[
N^J\sim J\log J
\ \Longleftrightarrow\
J^{-\frac{m}{n}}\sim (N^J)^{-\frac{m}{n}}(\log N^J)^{\frac{m}{n}} 
\]
completes the proof.
\end{proof}

%\section*{References}


\begin{thebibliography}{10}

\bibitem{CS90}
B.~Carl and I.~Stephani.
\newblock {\em Entropy, Compactness and the Approximation of Operators}.
\newblock Cambridge Univ. Press, 1990.

\bibitem{EE87}
D.E. Edmunds and W.D. Evans.
\newblock {\em Spectral Theory and Differential Operators}.
\newblock Oxford Univ. Press, 1987.

\bibitem{EE04}
D.E. Edmunds and W.D. Evans.
\newblock {\em Hardy Operators, Functions Spaces and Embeddings}.
\newblock Springer, 2004.

\bibitem{ET96}
D.E. Edmunds and H.~Triebel.
\newblock {\em Function Spaces, Entropy Numbers, Differential Operators}.
\newblock Cambridge Univ. Press, 1996.

\bibitem{EH93}
W.D. Evans and D.J. Harris.
\newblock Fractals, trees and the {N}eumann {L}aplacian.
\newblock {\em Math. Ann.}, 296:493--527, 1993.

\bibitem{HT94}
D.D. Haroske and H.~Triebel.
\newblock Entropy numbers in weighted function spaces and eigenvalue
  distributions of some degenerate pseudodifferential operators. {I}.
\newblock {\em Math. Nachr.}, 167:131--156, 1994.

\bibitem{HT08}
D.D. Haroske and H.~Triebel.
\newblock {\em Distributions, Sobolev Spaces, Elliptic Equations}.
\newblock EMS Publishing House, 2008.

\bibitem{Mi15b}
T.~Mieth.
\newblock Compact embeddings of sobolev spaces with power weights perturbed by
  slowly varying functions.
\newblock {\em submitted}.

\bibitem{Mi15a}
T.~Mieth.
\newblock Entropy and approximation numbers of embeddings of weighted sobolev
  spaces.
\newblock {\em J. Approx. Theory}, 192:250--272, 2015.

\bibitem{Pi78}
A.~Pietsch.
\newblock {\em Operator Ideals}.
\newblock VEB Deutscher Verlag der Wissenschaften, 1978.

\bibitem{Pi87}
A.~Pietsch.
\newblock {\em Eigenvalues and s-Numbers}.
\newblock Akademische Verlagsgesellschaft Geest \& Portig K.-G., 1987.

\bibitem{Tri92}
H.~Triebel.
\newblock {\em Higher Analysis}.
\newblock Barth, 1992.

\bibitem{Tri94}
H.~Triebel.
\newblock Relations between approximation numbers and entropy numbers.
\newblock {\em J. Approx. Theory}, 78:112--116, 1994.

\bibitem{Tri12}
H.~Triebel.
\newblock Entropy and approximation numbers of limiting embeddings, an approach
  via {H}ardy inequalities and quadratic forms.
\newblock {\em J. Approx. Theory}, 164(1):31--46, 2012.

\end{thebibliography}
\end{document}